\documentclass[a4paper,12pt]{amsart}
\usepackage[english]{babel}
\usepackage{amsmath,amssymb,amsthm}
\usepackage{calrsfs}

\usepackage[pdftex]{color}

\usepackage[bookmarks=true,hyperindex,pdftex,colorlinks,citecolor=blue,urlcolor=blue]{hyperref}

\theoremstyle{plain}
\newtheorem{theorem}{Theorem}[section]
\newtheorem{corollary}[theorem]{Corollary}
\newtheorem{prop}[theorem]{Proposition}
\newtheorem{lemma}[theorem]{Lemma}

\theoremstyle{definition}

\newtheorem{remark}[theorem]{Remark}

\newtheorem{definition}[theorem]{Definition}

\newtheorem{problem}[theorem]{Problem}

 \DeclareMathOperator{\dist}{dist}
 \DeclareMathOperator{\supp}{supp}
 \DeclareMathOperator{\sign}{sign}
 \DeclareMathOperator{\Lin}{span}

\newcommand{\R}{\mathbb{R}}
\newcommand{\N}{\mathbb{N}}

\newcommand{\D}{\mathbb{D}}
\newcommand{\Q}{\mathbb{Q}}

\newcommand{\U}{\mathfrak U}

\newcommand{\Norm}{|\mkern-2mu|\mkern-2mu|}

\newcommand{\eps}{\varepsilon}

\renewcommand{\leq}{\leqslant}
\renewcommand{\geq}{\geqslant}
\renewcommand{\le}{\leqslant}

\DeclareMathOperator{\NA}{NA}

\begin{document}
\title[Extremely nonlineable set of norm attaining functionals]{Equivalent norms with an extremely nonlineable set of norm attaining functionals}

\author[V.~Kadets]{Vladimir Kadets}
\author[G.~L\'opez]{Gin\'es L\'opez}
\author[M.~Mart\'{\i}n]{Miguel Mart\'{\i}n}
\author[D.~Werner]{Dirk Werner}

\address[Kadets]{School of Mathematics and Informatics \\
  V.~N.~Karazin   Kharkiv National University \\
pl.~Svobody~4 \\
61022~Kharkiv \\ Ukraine
\newline
\href{http://orcid.org/0000-0002-5606-2679}{ORCID: \texttt{0000-0002-5606-2679} }
}
\email{vova1kadets@yahoo.com}

\address[L\'opez]{Departamento de An\'{a}lisis Matem\'{a}tico \\ Facultad de
 Ciencias \\ Universidad de Granada \\ 18071 Granada, Spain
\newline
\href{http://orcid.org/0000-0002-3689-1365}{ORCID: \texttt{0000-0002-3689-1365} }
 }
\email{glopezp@ugr.es}

\address[Mart\'{\i}n]{Departamento de An\'{a}lisis Matem\'{a}tico \\ Facultad de
 Ciencias \\ Universidad de Granada \\ 18071 Granada, Spain
\newline
\href{http://orcid.org/0000-0003-4502-798X}{ORCID: \texttt{0000-0003-4502-798X} }
 }
\email{mmartins@ugr.es}

\address[Werner]{Department of Mathematics \\ Freie Universit\"at Berlin \\
Arnimallee~6 \\ D-14195~Berlin \\ Germany\newline
\href{http://orcid.org/0000-0003-0386-9652}{ORCID: \texttt{0000-0003-0386-9652}}
}
\email{werner@math.fu-berlin.de}

\thanks{The work of the first-named author
was performed during his visit to Freie Universit\"at Berlin in the framework of a grant from the {\it Alexander-von-Humboldt Stiftung}. The research of the second and third authors is partially supported by Spanish MINECO/FEDER grant MTM2015-65020-P}

\begin{abstract}
We present a construction that enables one to find Banach spaces $X$
whose sets $\NA(X)$ of norm attaining functionals do not contain
two-dimensional subspaces and such that, consequently, $X$ does not
contain proximinal subspaces of finite codimension greater than one,
extending the results recently provided by Read \cite{Read} and
Rmoutil \cite{Rmoutil}. Roughly speaking, we construct an equivalent
renorming with the requested properties for every Banach space $X$ where the set $\NA(X)$ for the original norm is not ``too large''.  The construction can be applied to every Banach space containing $c_0$ and having a countable system of norming functionals, in particular, to separable Banach spaces containing $c_0$. We also provide some geometric properties of the norms we have constructed.
\end{abstract}

\date{September 5th, 2017; Revised version; October 12th, 2017}



\maketitle
\thispagestyle{empty}

\section{Introduction}
A subset $Y$ of a (real) Banach space $X$ is said to be proximinal if
for every $x \in X$ there is a $y \in Y$ such that $\|x - y\| =
\dist(x,Y)$. The classical Bishop-Phelps theorem implies that every
infinite-dimensional  Banach space contains a one-codimensional
proximinal subspace. More than 40 years ago, Ivan Singer
\cite[Problem~2.1]{Singer}
asked whether every infinite-dimensional Banach space contains
proximinal subspaces of codimension~$2$. Recently  Charles J.~Read
\cite{Read} answered this question in the negative. The corresponding space $\mathcal R$ is  $c_0$ equipped with a special equivalent  norm  $\Norm\cdot\Norm$ ingeniously constructed by Read.

In  \cite[Theorem~4.2]{Rmoutil}, Martin Rmoutil demonstrates that the
same space $\mathcal R$ gives the negative solution to another (at
that time open) problem by Gilles Godefroy
\cite[Problem~III]{Godefroy}: is it true that  for every
infinite-dimensional Banach space the set of those functionals in the
dual space which attain their norm contains a two-dimensional linear
subspace?
Recall that a subset $S$ of a vector space is called \emph{lineable} if
$S\cup \{0\}$ contains an infinite-dimensional linear subspace, and we
call it \emph{extremely nonlineable} if $S\cup\{0\}$ does not even
contain a two-dimensional subspace. So by
Rmoutil's work, the set of norm attaining functionals on $\mathcal{R}$
is extremely nonlineable.

We note that there is a general statement saying that if $X$ contains proximinal subspaces of finite codimension at least two, then the set of norm attaining functionals contains a two-dimensional linear subspace (see \cite[Proposition~III.4]{Godefroy}).

Motivated by these facts, let us say that an equivalent norm $p$ on a
Banach space $X$ is a \emph{Read norm} if the set of norm attaining
functionals for this norm  does not contain two-dimensional linear
subspaces, so the space $X$ endowed with the norm $p$
  does not contain proximinal subspaces of finite codimension greater than
  one. In this paper we present a clear geometric idea which enables
us to simplify substantially Read's original construction of a
Read norm on $c_0$, and to extend the construction to some other
spaces. In particular, we show that every Banach space having a
countable norming system of functionals and containing a copy of $c_0$
admits an equivalent Read norm. We further provide some geometric
properties of the constructed Read norms which extend
the ones given in \cite{KadetsLopezMartin} for   Read's original
space~$\mathcal R$.

To this end, we introduce the concept of modesty and weak-star modesty of subspaces (see Definition~\ref{def-operator range+modest}) and show that a Read norm can be constructed whenever the linear span of the set of norm attaining functionals is weak-star modest.

The outline of the paper is as follows. We finish this introduction
with a subsection which collects all the notation and terminology used
in the paper. We devote Section~\ref{sect:preliminaries} to
preliminaries: we provide properties of two kinds of renorming of a
Banach space which will be used throughout the paper,  we introduce
the concept of modest and weak-star modest subspace, and we give some
needed results. The main part of the paper is contained in
Section~\ref{sect:main-construction}, where we show that a Banach
space admits an equivalent Read norm if the linear span of the set of
norm attaining functionals for the given norm is weak-star modest in
the dual space, recovering in particular the original results of Read
and Rmoutil. We also show that the constructed Read norms are always
strictly convex. Section~\ref{sect:containing_c_0} contains the main
application of the previous result: if a Banach space $X$ has a
countable norming system of functionals and contains an isomorphic
copy of $c_0$, then it admits an equivalent Read norm;
in particular, this is so if $X$ is separable and contains a copy
of~$c_0$. We also show that for every $0<\eps<2$, an equivalent Read
norm can be chosen in such a way that all convex combinations of
slices of its unit ball have diameter greater than $2-\eps$, so its
dual norm is $(2-\eps)$-rough; in the case when $X$ is separable, it
is possible to get a Read norm which is strictly convex and smooth and
whose dual norm is strictly convex and rough; if moreover $X^*$ is
separable, then in addition to  the above properties the bidual norm
is strictly convex. Finally, we discuss in Section~\ref{sect-limits}
some limitations of our construction as, for instance, that no Banach
space with the Radon-Nikod\'{y}m property admits an equivalent norm
for which the linear span of the norm attaining functionals is
weak-star modest.

\subsection{Notation and terminology}
Throughout the paper, the letters $X$, $Y$, $Z$ will stand for real Banach spaces. For a Banach space $X$, $X^*$ denotes its topological dual, $B_X$ and $S_X$ are, respectively, the closed unit ball and the unit sphere of $X$, and we write $J_X\colon X\longrightarrow X^{**}$ to represent the canonical isometric inclusion of $X$ into its bidual. We write $\NA(X)$ to denote the subset of $X^*$ of all functionals attaining their norm, that is, those functionals $f\in X^*$ such that $\|f\|=|f(x)|$ for some $x\in S_X$. If necessary, we will write $\NA(X,\|\cdot\|)$ to make clear that we are considering the space $X$ endowed with the norm $\|\cdot\|$.

A Banach space $X$ (or its norm) is said to be \emph{strictly convex} if $S_X$ does not contain any non-trivial segment or, equivalently, if $\|x+y\|<2$ whenever $x,y\in B_X$, $x\neq y$. The space $X$ is said to be \emph{smooth} if its norm is G\^{a}teaux differentiable at every non-zero element. A norm of a Banach space $X$ is said to be $\rho$-\emph{rough} ($0<\rho\leq 2$) if
$$
\limsup_{\|h\|\to 0}\frac{\|x+h\|+\|x-h\|-2\|x\|}{\|h\|}\geq \rho
$$
for every $x\in X$. We refer the reader to the classical books \cite{D-G-Z} and
\cite{FHHMPZ} for more information and background on the geometry of Banach spaces.

Finally, we will denote by $\{e_n\}$ the canonical basis of $c_0$ or $\ell_1$, that is, the $k$-th coordinate of $e_n$ equals 0 for $n \neq k$ and equals $1$ for $n = k$.

\section{Preliminaries}\label{sect:preliminaries}
Our first goal in this section is to present the properties of two
types of equivalent renormings of a Banach space. In the first one, we
add to the original norm of each element the norm of its image under
the action of a fixed operator. This  kind of renorming is well known in Banach space theory, see e.g.\   \cite[Proposition~III.2.11]{HWW},  and it  was also used by
Read to produce his counterexample \cite{Read}.

\begin{lemma}\label{lemma-prelim-sumofnorms}
Let $X$, $Y$ be Banach spaces, and let $R\colon X\longrightarrow Y$ be a bounded linear operator. Define an equivalent norm on $X$ by
$$
\Norm x\Norm =\|x\|_X + \|R(x)\|_Y \qquad (x\in X).
$$
Then
\begin{enumerate}
\item[(a)] $\displaystyle B_{(X,\Norm \cdot\Norm )^*}=B_{(X,\|\cdot\|)^*}+R^*(B_{Y^*})$;
\item[(b)] $\Norm x^{**}\Norm =\|x^{**}\|_{X^{**}} + \|R^{**}x^{**}\|_{Y^{**}}$ for every $x^{**}\in X^{**}$;
\item[(c)] if $x\in X$ and $x^*\in X^*$ satisfy  $\Norm x^*\Norm =1$ and $x^*(x)=\Norm x\Norm $, then
$x^*=\tilde{x}^* + R^*y^*$ where $\tilde{x}^*\in S_{(X,\|\cdot\|)^*}$ with $\tilde{x}^*(x)=\|x\|_X$ and $y^*\in S_{Y^*}$ with $y^*(Rx)=\|Rx\|_Y$;
\item[(d)] if $R(X)$ is strictly convex and $R$ is one-to-one, then $(X,\Norm \cdot\Norm )$ is strictly convex;
\item[(e)] if $X$ is $\rho$-rough for some $0<\rho\leq 2$, then $(X,\Norm \cdot\Norm )$ is $\rho(1+\|R\|)^{-1}$-rough.
\end{enumerate}
\end{lemma}

\begin{proof}
(a) Write $D=B_{(X,\|\cdot\|)^*}+R^*(B_{Y^*})$.  First, it is clear that
$$
\sup_{x^* \in D} x^*(x) = \Norm x\Norm
$$
for every $x \in X$. This means that $B_{(X,\Norm \cdot\Norm )}$ is the polar set of $D$. Consequently, $B_{(X,\Norm \cdot\Norm )^*}$ is the bipolar of $D$. So, it remains to demonstrate that $D$ is weak-star closed, a fact which follows from the fact that both $B_{(X,\|\cdot\|)^*}$ and $R^*(B_{Y^*})$ are weak-star compact.

(b) This is just \cite[Proposition~3]{KadetsLopezMartin}. Remark that
this fact can be  deduced much more easily  directly from~(a).

(c) This is immediate from~(a).

(d) This fact is widely used in the theory of equivalent renormings; it was first remarked by Victor Klee, see the proof of \cite[Ch.~4,  \S~2, Theorem~1]{Diest-Geom}.

Finally, (e) follows immediately from the definition of roughness.
\end{proof}

In the second type of renorming, the new unit ball is the sum of the
given unit ball and the image of a weakly compact unit ball by a
bounded linear operator. This kind of renorming was used in \cite{DebsGodSR} to study properties of the set of norm attaining functionals.

\begin{lemma}\label{lemma-prelim-sum-of-balls}
Let $X$ be a Banach space, let $Z$ be a reflexive space and let $S\colon Z\longrightarrow X$ be a bounded linear operator. Then there is an equivalent norm $|\cdot|$ on $X$ whose unit ball is the set $B_X + S(B_Z)$, and the following assertions hold:
\begin{enumerate}
  \item[(a)] $|x^*|=\|x^*\|_{X^*} + \|S^*x^*\|_{Z^*}$ for every $x^*\in X^*$;
  \item[(b)] $\displaystyle B_{(X,|\cdot|)^{**}} = B_{(X,\|\cdot\|)^{**}} + J_X( S(B_{Z}))$;
  \item[(c)] if $X$ and $Z$ are strictly convex, then $(X,|\cdot|)$ is strictly convex;
  \item[(d)] $\NA(X,|\cdot|)=\NA(X,\|\cdot\|)$.
\end{enumerate}
\end{lemma}

\begin{proof}
First, as $Z$ is reflexive and $S$ is weakly continuous,  the set
$S(B_Z)$ is weakly compact, so $B_X + S(B_Z)$ is closed. As it is also
bounded, balanced and solid, it is the unit ball of an equivalent norm $|\cdot|$ on $X$.

(a) is elementary and it is shown in the proof of \cite[Theorem~9(4)]{DebsGodSR}.

(b) follows from (a) of this lemma and (a) of the previous Lemma~\ref{lemma-prelim-sumofnorms}.

(c) Consider $\widetilde{x},\widetilde{y}\in S_{(X,|\cdot|)}$ such that $|\widetilde{x} + \widetilde{y}|=2$. Write $\widetilde{x}=x+T(u)$, $\widetilde{y}=y+T(v)$ with $x,y\in B_X$ and $u,v\in B_Z$ and consider $f\in X^*$ with
$$
|f|=1 \quad \text{and} \quad |f(\widetilde{y} + \widetilde{z})|=2.
$$
As we have that
\begin{align*}
2=|f(\widetilde{x}+\widetilde{y})|
& =|f(x+y)+[T^*f](u+v)| \\ &\leq \|f\|_{X^*}\|x+y\| +
\|T^*f\|_Y\|u+v\| \\
&\leq 2(\|f\|_{X^*}+\|T^*f\|)=2,
\end{align*}
it follows that $\|x+y\|=2$ and $\|u+v\|=2$. Since $X$ and $Z$ are both strictly convex, it follows that $x=y$ and $u=v$, so $\widetilde{x}=\widetilde{y}$.

(d) is proved in \cite[Theorem~9(4)]{DebsGodSR}: a bounded linear functional attains its supremum on $B_{(X,|\cdot|)}$ if and only if it attains its supremum on both $B_X$ and $S(B_Z)$, but all functionals attain their maxima on the weakly compact set $S(B_Z)$.
\end{proof}

The second goal in this section is to introduce the concepts of
modesty and weak-star
modesty of subspaces of a Banach space and to present some properties which will be important in our further discussion.

\begin{definition} \label{def-operator range+modest}
A linear subspace $Y$ of a Banach space $X$ is said to be an \emph{operator range} if there is an infinite-dimensional Banach space $E$ and a bounded injective operator $T\colon  E \longrightarrow X$ such that $T(E) = Y$. A linear subspace $Z \subset X$  is said to be \emph{modest} if there is a separable dense operator range $Y \subset X$ such that $Z \cap Y = \{0\}$. If $X$ is a dual space, a  linear subspace $Z \subset X$  is said to be \emph{weak-star modest} if there is a separable weak-star dense operator range $Y \subset X$ such that $Z \cap Y = \{0\}$.
\end{definition}

The study of dense operator ranges in Hilbert spaces goes back to
Dixmier, and many results were given by Fillmore and Williams (see
\cite{Fill-Will}). The extension of this study to operator ranges in
Banach spaces has attracted the attention of many mathematicians since
the domain of a closed operator between Banach spaces is an operator
range and every operator range is the domain of some closed linear
operator. We refer to the paper \cite{CrossOstrovskiiShevchik} (and
references therein) for a detailed account of the known results about
operator ranges and also for references and background.

We would like
to emphasize some remarks. Let $X$ be a Banach space and let $Y$ be a
linear subspace. First, $Y$ is an operator range if and only if there
is a complete norm on $Y$ which is stronger than the restriction  of
the given norm of $X$ to $Y$, see \cite[Proposition~2.1]{Cross}; if $Y$ is dense, $Y$ is contained in a non-closed dense operator range if and only if it is non-barrelled, see \cite[Theorem~15.2.1]{Wilansky}; finally, the injectivity of $T$ in the definition of operator range can be substituted by the condition $\dim Y = \infty$, because for every non-injective $T\colon  E \longrightarrow X$ there is an injective $\widetilde T\colon  E/ \ker T \longrightarrow X$ with the same range.

Next, we would like to make some remarks about modest and weak-star
modest subspaces. The first observation is that in the definition of
modest (and weak-star modest) subspace, the space $E$ which is the
domain of $T$ can be supposed to be separable (just consider the
closed linear span of the inverse image of a dense subset of
$Y$). Actually, any  infinite-dimensional separable Banach space can be chosen to be the domain of the dense (or weak-star dense) operator range,
because every separable infinite-dimensional Banach space can be
densely and injectively embedded into any other separable
infinite-dimensional Banach space, and we may even suppose that the
operator $T$ is nuclear, see
\cite[Proposition~3.1]{CrossOstrovskiiShevchik} for both
results. We will often apply this remark in that the (weak-star) modesty of
  $Z\subset X$ can be witnessed by an operator range $Y=T(\ell_1)$. Remark also that, obviously, if a subspace is modest (or weak-star modest) then all smaller subspaces are also modest (or weak-star modest).

Here is the key example of a modest subspace.

\begin{prop}\label{modest-in-ell1}
The subspace $\Lin\{e_n\} \subset \ell_1$ consisting of all sequences
with finite support is modest.
\end{prop}

This results immediately from the following lemma which will also be
useful later on.

\begin{lemma} \label{Lem-op-im-inell1}
There is a dense operator range $Y \subset \ell_1$ such that every
non-zero element of $Y$ has a finite number of zero coordinates.
\end{lemma}

\begin{proof}
Let $\D$ be the closed unit disc, and let $A(\D)$ be the disc algebra
consisting of all continuous  functions on $\D$ that are analytic on
the interior of $\D$, viewed as a real Banach space. We let $A_r(\D) \subset A(\D)$ be the closed real subspace consisting of those $f$ that take real values on the real axis, and denote $t_n = 2^{-n}$ for every $n\in \N$. We define $T\colon A_r(\D)\longrightarrow
\ell_1$ by
$$
Tf = \left(f(t_1), \frac12 f(t_2), \frac14 f(t_3), \ldots \right)
\qquad \bigl(f\in A_r(\D)\bigr).
$$
Then, the identity theorem for analytic functions
implies that in $Y := T(A_r(\D))$ every non-zero element has a finite
number of zero coordinates (if any). It remains to demonstrate the
density of $Y$ in $\ell_1$. To this end it is sufficient to show that
every element $e_m$ of the canonical basis of $\ell_1$ belongs to the
closure of $Y$. Indeed, for a fixed $m\in \N$, consider the function
$f(z) = \frac{1}{4}(4 - (z - t_m)^2)$ for every $z\in \D$. This $f \in
A_r(\D)$ takes the value $1$ at $t_m$ and $ 0 < f(t_k) < 1 $ for all
$k \neq m$. Denote $f_n = f^n \in A_r(\D)$. Then,
$\lim_{n\to\infty}f_n(t_m) = 1$, and $\lim_{n\to\infty}f_n(t_k) = 0$
for $k \neq m$, indeed $\lim_{n\to\infty} \sup_{k\neq m} |f_n(t_k)| =0$, hence
$\lim_{n\to\infty}Tf_n = e_m/2^{m-1}$, and $e_m$ is in the closure of~$Y$.
\end{proof}

In fact, Proposition~\ref{modest-in-ell1} can be generalised.

\begin{prop} \label{prop countable-modest}
For every separable Banach space $X$ every subspace with a countable
Hamel basis is modest.
\end{prop}

\begin{proof}
For every subspace $W$ with a countable Hamel basis, there is a dense
subspace $W_1 \supset W$ with a countable Hamel basis, say
$\{w_n\colon n\in \N\}$. The construction of \cite[Proposition~1.f.3]{Lin-Tza-I} provides us with sequences $(v_n)$ in $X$ and $(v_n^*)$ in $X^*$ such that
$\Lin\{v_1,\dots,v_n\} = \Lin\{w_1,\dots,w_n\}$ and $v_n^*(v_m) =
\delta_{m,n}$ for all $m$ and~$n$. Upon replacing $v_n$ by
$v_n/\|v_n\|$ and $v_n^*$ by $\|v_n\|v_n^*$ we may assume that $(v_n)$ is
bounded.

Consider now the bounded linear operator $S\colon \ell_1\longrightarrow X$
defined by $S(x)=\sum_{n=1}^\infty x(n) v_n$ for every $x\in
\ell_1$ and  let $T\colon A_r(\D)\longrightarrow \ell_1$ be the operator
defined in Lemma~\ref{Lem-op-im-inell1}. Then,
$\widetilde{T}=S\circ T\colon A_r(\D)\longrightarrow X$ is bounded, has
dense range since $T(A_r(\D))$ is dense in $\ell_1$ and $S(\ell_1)
\supset W_1$ is dense in $X$. Finally, $\widetilde{T}(A_r(\D)) \cap
W_1=\{0\}$. Indeed, if $w= \sum_{k=1}^n \alpha_k v_k \in W_1$ has the form
$\sum_{k=1}^\infty \frac{f(t_k)}{2^{k-1}} v_k$, then apply $v_l^*$ to
these series to see that $\alpha_l= f(t_l)/2^{l-1}$ for $l\le n$ and
$0= f(t_l)/2^{l-1}$ for $l> n$. The latter implies $f=0$ since $f$ is
analytic, therefore $w=0$.

This shows that $W_1$ is modest and so is the smaller
subspace $W$.
\end{proof}

We would like to comment that the gist of the construction of a
Markushevich basis in
\cite[Proposition~1.f.3]{Lin-Tza-I} alluded to above is the
Gram-Schmidt orthogonalisation. Indeed, let $J\colon X\longrightarrow
H$ be an injective bounded linear operator into a Hilbert space with
dense range; for example, embed $X$ isometrically into $C[0,1]$ and
further continuously into $L_2[0,1]$, and let $H\subset L_2[0,1]$ be
the closure of the image of $X$ in $L_2[0,1]$. Then perform the
Gram-Schmidt procedure on the linearly independent sequence $(J(w_n))$
to obtain an orthogonal basis $(h_n)\subset J(X)$ for
$H=H^*$. Finally, put $v_n= J^{-1}(h_n)$ and $v_n^*=J^*(h_n)$,
i.e., $v_n^*(x)= \langle Jx, h_n\rangle_H$.

We next present a known result about operator ranges which we will use
later on.

\begin{prop}[\mbox{\cite[Proposition~2.6]{Cross}}]\label{two operator ranges}
In every separable infinite-dimensional Banach space $X$ there are two dense operator ranges with trivial intersection.
\end{prop}

The main property of operator ranges which we will need in the paper is the following one.

\begin{prop}  \label{prop-modest-dense}
Let  $Y \subset X$ be a separable operator range. Then, there is an injective norm-one linear operator $T\colon  \ell_1 \longrightarrow Y$ such that the set $\left\{\frac{Te_n}{\|Te_n\|}\colon n \in \N \right\}$ is dense in $S_Y$.
\end{prop}

We need the following technical result to provide the proof of the proposition.

\begin{lemma}  \label{lemma-modest-coefficients-abs-conv}
Let $X$ be a Banach space and let  $Y \subset X$ be an operator range. Then, for every sequence $\{u_n\}$ in $Y$, there is a sequence of positive reals $\{s_n\}$ in $(0,1]$ such that for every $x=(x_1, x_2, \ldots) \in \ell_1$ we have $\sum_n  s_n x_n u_n \in Y$.
\end{lemma}

\begin{proof}
By definition, there is a Banach space $E$ and a bounded bijective
linear operator $T\colon  E \longrightarrow Y$. To complete the proof it is
sufficient to take $s_n =  \min\{1, \|T^{-1}u_n\|^{-1}\}$ and remark
that the series  $\sum_n s_n x_n T^{-1} u_n$ converges absolutely for
each $x\in\ell_1$, say to  $e \in E$, so $\sum_n  s_n x_n u_n = T(e) \in Y$.
\end{proof}

\begin{proof}[Proof of Proposition~\ref{prop-modest-dense}]
By the remarks after Definition~\ref{def-operator range+modest}, there
is an infinite-dimensional separable Banach space $E$ and a bounded
injective linear operator $T_1\colon  E \longrightarrow Y$ with dense range.
Applying Proposition~\ref{two operator ranges}, we can find two dense
operator ranges $E_1, E_2 \subset E$ with trivial
intersection. Without loss of generality, we may assume the existence
of injective $U_1, U_2\colon \ell_1\longrightarrow E$ such that
$U_i(\ell_1) = E_i$, $i=1,2$ (see the remarks following
Definition~\ref{def-operator range+modest}).
Fix a countable dense subset $\{w_n\}_{n \in \N} \subset S_{T_1(E_1)}$, then
$\{w_n\}_{n \in \N}$ is dense in $S_Y$ as well.  Denote $u_n =
\frac{T_1^{-1}(w_n)}{\|T_1^{-1}(w_n)\|} \in E_1$, select the
corresponding sequence $\{s_n\}$ from
Lemma~\ref{lemma-modest-coefficients-abs-conv} and define the
requested operator $T\colon  \ell_1 \longrightarrow Y$ as follows:
$$
T(e_n) = s_n T_1\left( u_n + \eps_n U_2 e_n\right),
$$
where the $\eps_n \in (0,1)$ are small enough to ensure that $\left\| T_1^{-1}(w_n) \right\|   \eps_n \longrightarrow 0$. Then
\begin{align*}
\left\|\frac{Te_n}{\|Te_n\|} - w_n \right\| & =
\left\|\frac{ T_1\left( u_n + \eps_n U_2 e_n\right)}{\| T_1\left( u_n + \eps_n U_2 e_n\right)\|} - w_n \right\|
\\ &=
\left\|\frac{ w_n +\eps_n  \|T_1^{-1}(w_n)\| T_1(U_2 e_n)}{\bigl\|  w_n
    +\eps_n  \|T_1^{-1}(w_n)\| T_1(U_2 e_n) \bigr\|} - w_n \right\|
\xrightarrow[n\to\infty]\, 0.
\end{align*}
So,  $\left\{\frac{Te_n}{\|Te_n\|}\colon n \in \N \right\}$ is dense
in $S_Y$.
It remains to demonstrate that $T$ is injective. Assume that for some $x = (x_1, x_2, \ldots) \in \ell_1$
$$
Tx = \sum_{n \in \N} x_n Te_n = \sum_{n \in \N}  T_1\left(x_n  s_n u_n\right) +  T_1U_2 \left (\sum_{n \in \N} x_n s_n \eps_n e_n\right) =  0.
$$
Then,
$$
\sum_{n \in \N}  T_1\left(x_n s_n u_n\right)   = - T_1U_2 \left (\sum_{n \in \N} x_n s_n \eps_n e_n\right) ,
$$
and by the injectivity of $T_1$
$$
\sum_{n \in \N}  x_n  s_n u_n   = - U_2 \left (\sum_{n \in \N} x_n s_n \eps_n e_n\right).
$$
But the left hand side of the last equation belongs to $E_1$ and the
right hand side belongs to $E_2$, so both of them are equal to
$0$. Since $\{e_n\}_{n \in \N}$ forms a basis of $\ell_1$ and $U_2$ is injective, this implies that $x = 0$. Finally, the fact that $\|T\|=1$ can be obtained just by
dividing by its norm.
\end{proof}

Our last result in this section allows us to extend a modest subspace from a complemented subspace to the whole space, in some cases.

\begin{prop}  \label{prop_modest_sum}
Let $X$ be a Banach space such that $X = X_1 \oplus X_2$ for suitable
closed subspaces $X_1$ and $X_2$.
Writing $X^*$ in its canonical form   $X^* = X_1^* \oplus X_2^*$ we have the following.
\begin{enumerate}
  \item[(a)] If $X_1^*$ is weak-star separable and $F_2 \subset X_2^*$ is weak-star modest in $X_2^*$, then  $X_1^* \oplus  F_2$ is weak-star modest in $X^*$.
  \item[(b)] If $X_1^*$ is norm separable and $F_2 \subset X_2^*$ is modest in $X_2^*$, then  $X_1^* \oplus  F_2$ is modest in $X^*$.
\end{enumerate}
\end{prop}

\begin{proof}
Let $P_1$, $P_2$ be the natural projections of $X^*$ onto $X_1^*$ and $X_2^*$,
respectively. For (a), take in $S_{X_1^*}$ a countable subset
$\{y_n^*\}_{n \in \N}$ whose linear span is weak-star dense in $X_1^*$;
for (b), take in $S_{X_1^*}$ a countable subset $\{y_n^*\}_{n \in \N}$
whose linear span is norm dense in $X_1^*$. Let $T\colon  \ell_1
\longrightarrow X_2^*$ be an injective operator whose image is weak-star
dense for the case~(a) and norm dense for the case~(b) in $X_2^*$ and
$F_2 \cap T(\ell_1) = \{0\}$. Without loss of generality we may assume
that $\|T(e_n)\| \longrightarrow 0 $, where $e_n$ are the elements of
the canonical basis of $\ell_1$ (indeed, if not, just compose $T$ with
the operator $T_1 \colon \ell_1 \longrightarrow \ell_1$ that maps
$e_k$ to $e_k/k$ for $k =1,2, \dots$). Also, fix a partition of $\N$, $\N =
\bigsqcup_{n \in \N} A_n$, into a countable family of disjoint
infinite subsets. Now let us define the requested operator $\widetilde
T\colon  \ell_1 \longrightarrow X^*$ as follows:
$$
\widetilde T(x) =  \sum_{n \in \N}  \sum_{k \in A_n} x_k  (y_n^* +
T(e_k)) \qquad \bigl( x=(x_n)_n \bigr),
$$
i.e., $\widetilde T(e_k) = y_n^* + T(e_k)$ for all $k \in A_n$. Then the closure of $\widetilde T(\ell_1)$ contains all the functionals $y_n^*$, and consequently it contains also all $T(e_k)$, so $\overline {\widetilde T(\ell_1)} \supset \Lin\{y_n^*\colon n \in \N\} \oplus T(\ell_1)$ which in the case (a) is weak-star dense in $X^*$ and norm dense in the case (b). Injectivity of $\widetilde T$ follows from injectivity of $T$. It remains to demonstrate that $\widetilde T(\ell_1)$ has trivial intersection with $X_1^* \oplus  F_2$. Indeed, let $x_1^* + f_2 = \sum_{n \in \N}  \sum_{k \in A_n} x_k  (y_n^* + T(e_k))$ for some $x = (x_1, x_2, \ldots) \in \ell_1$ with $x_1^* \in X_1^*$ and $f_2 \in F_2$. Applying $P_2$, we obtain $f_2 =  \sum_{n \in \N}  \sum_{k \in A_n} x_k T(e_k) = Tx$, which means that $x = 0$.
\end{proof}

\section{The main construction}\label{sect:main-construction}

Our goal here is to present a general argument providing Read
norms. We also present some geometric properties of the norms
constructed in this way.
We denote the dual norm to an equivalent norm $p$ by~$p^*$.

\begin{theorem} \label{theor-main-construction}
Let $X$ be a Banach space such that $\Lin(\NA(X))$ is a weak-star
modest subspace of $X^*$. Then $X$ possesses  an equivalent Read
norm $p$. Moreover $p$ can be chosen in such a way that, given
two linearly independent functionals
$x^*, z^* \in \NA(X,p)$ with $p^*(x^*) = p^*(z^*) =1$,  one has $x^* +
z^* \notin  \NA(X,p)$ or   $x^* - z^* \notin  \NA(X,p)$.
\end{theorem}

\begin{proof}
Let $Y \subset X^*$ be a separable weak-star dense operator range with
$\Lin(\NA(X))\cap Y = \{0\}$ according to Definition~\ref{def-operator range+modest}. By Proposition~\ref{prop-modest-dense}, we may assume that  $Y = T(\ell_1)$, where $T\colon  \ell_1 \longrightarrow X^*$   is  an injective bounded linear operator such that the set $\left\{\frac{Te_n}{\|Te_n\|}\colon n \in \N \right\}$ is dense in $S_Y$. Take a sequence $\{r_n\}$ of positive reals such that $\sum_{k \in \N} r_k <\infty$, denote $v_n^* = T(e_n)$, and consider the operator
$R\colon X\longrightarrow \ell_1$ given by $[R(x)](n)=r_n v_n^*(x)$ for every $n\in \N$ and every $x\in X$. Then, we define an equivalent norm on $X$ by
\begin{equation*}
p(x)=\|x\| + \|Rx\|_{\ell_1} \qquad (x\in X).
\end{equation*}
The adjoint operator  $R^* \colon  \ell_\infty \longrightarrow X^*$ acts as follows: $R^*\bigl(\{t_n\}_{n \in \N}\bigr) = \sum_{n \in \N}  t_n r_n v_n^*$. Consequently, according to Lemma~\ref{lemma-prelim-sumofnorms}(a), we have that
\begin{equation*}
B_{(X,p)^*} = B_{X^*} + R^*(B_{\ell_\infty})=B_{X^*}+\sum_{n \in \N}r_n[-v_n^*, v_n^*].
\end{equation*}

Consider two linearly independent functionals $x^*, z^* \in \NA(X,p)$ with $p^*(x^*) = p^*(z^*) =1$, and let $x, z \in X$ with $p(x)=p(z)=1$ such that $x^*(x)=z^*(z)=1$. Due to Lemma~\ref{lemma-prelim-sumofnorms}(c), there are representations
\begin{equation} \label{eq-x*z^*}
x^* = x_0^* + \sum_{n \in \N}  t_n r_n v_n^*, \quad  z^* = z_0^* + \sum_{n \in \N}  \tau_n r_n v_n^*.
\end{equation}
with $t_k, \tau_k \in [-1,1]$ such that $x_0^*, z_0^*  \in S_{X^*} \cap \NA(X)$, for every $n \in \N$ where
$v_n^*(x) \neq 0$ one has $t_n = \sign  v_n^*(x)$, and  for every $n \in \N$ where $ v_n^*(z) \neq 0$ one has $\tau_n = \sign  v_n^*(z)$.
Let $\theta = \pm 1$ be a sign such  that  $x  \neq \theta z$. First,
remark that, by weak-star density of $Y$,
the set of restrictions of functionals from $Y$ to the
linear span of $x$ and $z$ is the whole $(\Lin\{x,z\})^*$. So, there
is $y_0^* \in S_Y$ such that $y_0^*(x) < 0$ and $y_0^*(\theta z) >
0$. Consequently, there is a neighbourhood $U_0$ of $y_0^*$ in $S_Y$
such that for all $y^* \in U_0$, we have $y^*(x) < 0$ and $y^*(\theta
z) > 0$. Then, for all those $n \in \N$ for which
$\frac{v_n^*}{\|v_n^*\|}\in U_0$, we have that
\begin{equation} \label{eq-tntaun}
t_n  +  \theta  \tau_n =  \sign  v_n^*(x) + \theta \sign  v_n^*(z) = 0.
\end{equation}

We are going to demonstrate that $x^* +   \theta z^* \notin
\NA(X,p)$. Assume to the contrary that there is $e \in X$ with
$p(e)=1$ at which  $x^* +   \theta z^* $ attains its norm, that is
$(x^* +   \theta z^*)(e) =p^*(x^* +   \theta  z^*)$.
Lemma~\ref{lemma-prelim-sumofnorms}(c) says that
one can write
\begin{equation} \label{eq-x*+z^*-no2}
\frac{x^* +   \theta  z^*}{ p^*(x^* +   \theta  z^*)} = h_0^*  +  \sum_{n \in \N}  s_n r_n v_n^* ,
\end{equation}
with $s_k \in [-1,1]$, $h_0^* \in \NA(X)$, and for every $n \in \N$ where $v_n^*(e) \neq 0$, one has $s_n = \sign v_n^*(e)$.

Since $Y$ is weak-star dense, it cannot be contained in a weak-star
closed hyperplane. Consequently, the set $S_Y \cap \{h^* \in X^*\colon
h^*(e) = 0 \} = S_Y \cap \{h^* \in Y\colon h^*(e) = 0 \}$ is nowhere
dense in  $S_Y$. This implies that there is a non-empty relatively
open subset $U_1 \subset U_0$ of $S_Y$ which does not intersect the
hyperplane $ \{h^* \in Y\colon h^*(e) = 0 \}$. Denote $N_1 = \left\{n
  \in \N\colon \frac{v_n^*}{\|v_n^*\|} \in U_1\right\}$, which is
non-empty  by density of $\{\frac{v_n^*}{\|v_n^*\|} \colon n\in\N\}$ in~$S_Y$.
Then, for every $n \in N_1$  the conditions \eqref{eq-tntaun} and the fact that $s_n = \sign v_n^*(e)$ hold true at the same time.

Now, from equations \eqref{eq-x*z^*} and \eqref{eq-x*+z^*-no2} we get
\begin{align*}
0 &= x^* +   \theta  z^* -  p^*(x^* +    \theta z^*) \frac{x^* +   \theta  z^*}{ p^*(x^* +    \theta z^*)} \\
&= (x_0^*+ \theta z_0^* -  p^*(x^* + \theta z^*)h_0^* )
+ \sum_{n \in \N} (t_n + \theta \tau_n  - p^*(x^* +    \theta z^*)s_n)  r_n v_n^*.
\end{align*}
In other words,
\begin{align*}
x_0^*+ \theta z_0^* -  p^*(x^* +  \theta z^*)h_0^*
= - T\left(\sum_{n \in \N} (t_n + \theta \tau_n  - p^*(x^* +    \theta z^*)s_n)  r_n e_n\right).
\end{align*}
The left hand side belongs to $\Lin(\NA(X))$, the right hand side belongs to $Y$, so both of them are equal to zero. Since $T$ is injective, and $\{e_n\}_{n \in \N}$ forms a basis of $\ell_1$, this means that all $t_n + \theta \tau_n  - p^*(x^* +    \theta z^*)s_n$ are equal to zero. On the other hand, as we remarked before, for every $n \in N_1$ we have $t_n  +  \theta  \tau_n = 0$ and $s_n = \sign v_n^*(e)\neq 0$, so $t_n + \theta \tau_n  - p^*(x^* + \theta z^*)s_n \neq 0$. This contradiction completes the proof.
\end{proof}

Observe that $\Lin(\NA(c_0))=\NA(c_0)$ consists on those elements of
$\ell_1$ that have finite support, so it is modest by
Proposition~\ref{modest-in-ell1}. Therefore,
Theorem~\ref{theor-main-construction} applies, giving Read's
\cite{Read} and Rmoutil's \cite{Rmoutil} results.

\begin{corollary}[\mbox{\cite{Read,Rmoutil}}]
There exists an equivalent norm $p$ on $c_0$ such that $\NA(c_0,p)$ does not contain two-dimensional subspaces and, therefore, $(c_0,p)$ does not contain finite-codimensional proximinal subspaces of codimension greater than~$1$.
\end{corollary}

Although  we extensively use Read's
ideas in our construction, his original  construction is not a particular
case of ours. Namely, Read's
norm on $c_0$ is defined by a very similar formula, but his choice
of corresponding functionals $v_n^*$ is quite different; in Read's
choice they  belong to $\NA(c_0)$ whereas our $v_n^*$ are sort of
``orthogonal'' to this set.

We will provide further examples in the next section.

Next, we would like to present some geometric properties of the Read norms we have constructed here, extending some of the results of \cite{KadetsLopezMartin}. First, we need to expound in detail the norms constructed in Theorem~\ref{theor-main-construction}.

\begin{remark}\label{remark-resume}
Let $X$ be a Banach space. If $\Lin(\NA(X))$ is a weak-star modest subspace of $X^*$, then there is a sequence $\{v_n^*\}_{n\in \N}$ in $B_{X^*}$ for which $\{v_n^*/\|v_n^*\|\}_{n\in \N}$ is weak-star dense in $S_{X^*}$, such that given a sequence $\{r_n\}_{n\in \N}$ of positive reals with $\rho=\sum_{k\in \N} r_k <\infty$ and defining the bounded linear operator $R\colon X\longrightarrow \ell_1$ by
\begin{equation}\label{eq:geom-def-operator-R}
[R(x)](n)=r_n v_n^*(x) \qquad \bigl(n\in \N,\ x\in X\bigr),
\end{equation}
the norm
\begin{equation}\label{eq:geom-def-Read-norm}
p(x)=\|x\|_X + \|R(x)\|_{\ell_1} \qquad (x\in X)
\end{equation}
is a Read norm.
If moreover $\Lin(\NA(X))$ is modest, we get that the sequence $\{v_n^*/\|v_n^*\|\}_{n\in \N}$ can be selected to be norm-dense in $S_{X^*}$.

Let us  mention that it is clear that $\|R\|\leq \rho$ and that $R$ is compact since $\|P_n R - R\| \leq \sum_{k>n} r_k \longrightarrow 0$, where $P_n$   projects $\ell_1$ onto $\Lin\{e_1,\dots,e_n\}$.
\end{remark}

We are now ready to present some geometric properties of our Read norms.

\begin{prop}\label{prop-geom-strictly-convex}
Let $X$ be a Banach space. If $\Lin(\NA(X))$ is a weak-star modest subspace of $X^*$, then the Read norm $p$ defined in \eqref{eq:geom-def-Read-norm} is strictly convex. Moreover, if $\Lin(\NA(X))$ is actually a modest subspace of $X^*$, then  $p$ can be built in such a way that $(X,p)^{**}$ is strictly convex and so $(X,p)^*$ is smooth.
\end{prop}

\begin{proof}
For the first part, we only have to show that the operator $R$ given
in \eqref{eq:geom-def-operator-R} is one-to-one and that $R(X)$ is
strictly convex, and then apply
Lemma~\ref{lemma-prelim-sumofnorms}(d). Both assertions are a
consequence of the fact that the sequence $\{v_n^*/\|v_n^*\|\}_{n\in
  \N}$ is weak-star dense in $S_{X^*}$, the first one being immediate. For the strict convexity of $R(X)$, consider $x,y\in X$ such that $R(x)\neq \alpha R(y)$ for every $\alpha>0$. Then, $x\neq \alpha y$ for every $\alpha>0$, so by the Hahn-Banach theorem, there is $x^*\in S_{X^*}$ such that $x^*(x)<0<x^*(y)$ and by weak-star density of $\{v_n^*/\|v_n^*\|\}_{n\in \N}$ in $S_{X^*}$, we get that there is $n\in \N$ such that $v_n^*(x)<0<v_n^*(y)$, so $|v_n^*(x+y)|<|v_n^*(x)|+|v_n^*(y)|$. From here, it is immediate that $\|R(x) + R(y)\|_{\ell_1} < \|R(x)\|_{\ell_1} + \|R(y)\|_{\ell_1}$, showing the strict convexity or $R(X)$.

For the moreover part,  we first use the modesty of $\Lin(\NA(X))$ in
order to get that $\{v_n^*/\|v_n^*\|\}_{n\in \N}$ is  norm-dense in $S_{X^*}$. By  Lemma~\ref{lemma-prelim-sumofnorms}(b), we know that the bidual norm of $p$ is given by
$$
p(x^{**})=\|x^{**}\|_{X^{**}} + \|R^{**}(x^{**})\|_{\ell_1^{**}} \qquad (x^{**}\in X^{**}).
$$
As $R$ is compact, $R^{**}(X^{**})\subset J_{\ell_1}(R(X))$, so to get the strict convexity of the bidual norm we only need to show that $R^{**}$ is one-to-one, but this is consequence of the fact that now the sequence $\{v_n^*/\|v_n^*\|\}_{n\in \N}$ is norm-dense in $S_{X^*}$, as this implies that $R^*(\ell_\infty)$ is norm dense in $X^*$.
\end{proof}

We do not know if for separable Banach spaces, the result above can be improved to get that the Read norm is actually weakly locally uniformly rotund, as it happens for the original Read norm of $c_0$ \cite[Theorem~9]{KadetsLopezMartin}.

We finish the section with the following result which appears in  \cite[Lemma~11]{KadetsLopezMartin}: given a Read norm on a separable Banach space, there is another equivalent Read norm which is smooth. One obtains this fact just applying the renorming sketched in Lemma~\ref{lemma-prelim-sum-of-balls}.

\section{Applicability of the main construction} \label{sect:containing_c_0}
The aim of this section is to demonstrate that
Theorem~\ref{theor-main-construction} is applicable (after making an
appropriate renorming) to all those Banach spaces that contain an
isomorphic copy of $c_0$ and have a countable norming system of
functionals. A \emph{countable norming system of functionals} of a
Banach space $X$ is a bounded subset $\{x_n^*\colon n\in \N\}$ of
$X^*$ for which there is a constant $K\geq 0$ such that
$$
\|x\|\leq K\,\sup_{n\in \N}\bigl|x_n^*(x)\bigr| \qquad (x\in X).
$$
Banach spaces with a countable norming system of functionals are those for which there is a bounded subset of the dual with non-empty interior which is weak-star separable or, equivalently, those which are isomorphic to closed subspaces of $\ell_\infty$, see \cite[p.~254]{DancerSims} for instance.

Our next result shows that the construction of the previous section is
applicable to all Banach spaces which are isomorphic to a closed
subspace of $\ell_\infty$ and contain a copy of $c_0$;
in particular, it is applicable to separable spaces containing a copy
  of~$c_0$.

\begin{theorem} \label{theor-main-application}
Let $X$ be a Banach space containing an isomorphic copy of $c_0$ and
possessing a countable norming system of functionals. Then $X$ is
isomorphic to a space $\widetilde X$ such that $\Lin(\NA(\widetilde
X))$ is weak-star modest in $\widetilde X^*$. Therefore, we can apply
Theorem~\ref{theor-main-construction} to get that the norm given by
\eqref{eq:geom-def-Read-norm} originating from the norm of $\widetilde X$
is a Read norm.
\end{theorem}

We need a preliminary technical result.

\begin{lemma} \label{lemma-c_0 in_ellinftyl}
Let $X$ be a Banach space containing an isomorphic copy of $c_0$ and possessing a countable norming system of functionals. Then $X$ is isomorphic to a closed subspace $X_1$ of $\ell_\infty$ containing the canonical copy of $c_0$ inside $\ell_\infty$.
\end{lemma}

\begin{proof}
As $X$ is isomorphic to a closed subspace of $\ell_\infty$, we can
assume that $X$ itself is a closed subspace of $\ell_\infty$. Denote
by $Y_1$ a closed subspace of $X$ that is isomorphic to
$c_0$. According to the Lindenstrauss-Rosenthal theorem
\cite[Theorem~2.f.12(i)]{Lin-Tza-I}, for arbitrary isomorphic closed
subspaces $Y_1, Y_2$ of $\ell_\infty$ such that both
$\ell_\infty/Y_1$,  $\ell_\infty/Y_2$ are non-reflexive, every
bijective isomorphism $T \colon Y_1 \longrightarrow Y_2$  extends to
an automorphism $\widetilde T\colon  \ell_\infty \longrightarrow
\ell_\infty$. If we apply this result to our $Y_1$, to $Y_2 = c_0$,
and to an arbitrary bijective isomorphism $T\colon  Y_1 \longrightarrow
c_0$ (which is possible by \cite[Proposition~2.f.13]{Lin-Tza-I}),
the resulting $X_1 = \widetilde T(X)$ will be the subspace we
are looking for.
\end{proof}

\begin{proof}[Proof of Theorem~\ref{theor-main-application}]
By Lemma~\ref{lemma-c_0 in_ellinftyl}, we may assume without loss of generality  that  $c_0 \subset  X \subset \ell_\infty$. Consider a non-trivial ultrafilter $\U$ on $\N$ and denote by $u$  the linear functional on $\ell_\infty$ that assigns to each $x = (x_n)_{n \in \N} \in \ell_\infty$ the $\U$-limit of its coordinates:
$$
u(x) = \lim_\U x_n.
$$
There are two cases: (1) for some non-trivial ultrafilter $\U$ our
space  $X$ lies in the corresponding $\ker u$,  and (2) $X
\not\subset  \ker u$ for any $\U$. Let us demonstrate that the second
case can be reduced to the first one. Indeed, in the second case denote by
$R_1\colon  \ell_\infty \longrightarrow \ell_\infty$ the right shift
operator: $R_1((x_1, x_2, \ldots)) = (0, x_1, x_2, \ldots)$. Then
always $R_1(X)  \not\subset \ker u$  (otherwise $X$ lies in the kernel
of the limit with respect to the shifted ultrafilter). Consequently,
$R_1(X)  \cap\ker u$  is a one-codimensional subspace of $R_1(X)\cong
X$, so    $\widetilde X:= \R e_1 \oplus (R_1(X)  \cap\ker u)$ is
isomorphic to $X$. Since  $c_0 \subset  \widetilde X \subset \ker u$,
the reduction to the first case is completed.

So the picture that we are considering is  $c_0 \subset X \subset \ker u$. Since $c_0$ forms an $M$-ideal of  $\ell_\infty$, $c_0$ is also an $M$-ideal of $X$  \cite[Proposition~I.1.17]{HWW}, that is, $X^* = (c_0)^\bot \oplus_1 \ell_1$. Then
$$
\NA(X) \subset \bigl[(c_0)^\bot \cap \NA(X)\bigr]  \oplus_1 \bigl[\ell_1 \cap \NA(X)\bigr] \subset (c_0)^\bot \oplus_1 \bigl[\ell_1 \cap \NA(X)\bigr],
$$
where in the first inclusion we use the elementary fact that if $f +
g$ with $\|f+g\| = \|f\| + \|g\|$ attains its norm, then both $f$ and
$g$ attain their norms.
If a non-zero element $f = (f_1, f_2, \ldots) \in \ell_1$  attains its
norm at some $x = (x_1, x_2, \ldots) \in S_X$, then for all $n$ where
$f_n \neq 0$  we have $|x_n| = 1$.  Since  $\lim_\U x_n = 0$, this
means that for every element $f \in \ell_1 \cap \NA(X)$ the set $\{n
\in \N\colon f_n =0 \}$ belongs to $\U$.  Any linear
combination of elements of $\ell_1 \cap \NA(X)$ will have the same
property.  Let $Y \subset \ell_1$ be the dense operator range from
Lemma~\ref{Lem-op-im-inell1}. Since $\ell_1$ is weak-star dense in
$X^*$, this $Y$ is also weak-star dense in $X^*$.  Every non-zero
element of $Y$ has a finite number of zero coordinates, but for $f\in Y \cap  \Lin(\NA(X))$, the number of zero coordinates is infinite by the above discussion.
Consequently $Y \cap  \Lin(\NA(X)) \subset Y \cap \Lin (\ell_1 \cap \NA(X)) = \{0\}$. This demonstrates that $\Lin(\NA(X))$ is weak-star modest in $X^*$.
\end{proof}

If $X$ is actually separable, things may be done in an easier fashion;
and  in the case when $X^*$ is separable we get a stronger result. We state the result here.

\begin{prop}\label{prop-separablecontainingc0}
Let $X$ be a separable Banach space containing $c_0$. Then, there is an equivalent norm $q$ on $X$ such that, in this norm, $(X, q)=c_0 \oplus_\infty Z$ for some $Z$ and  $\Lin(\NA((X, q))) \subset \NA(c_0) \oplus Z^*$ is weak-star modest. If moreover $X^*$ is separable, then $\Lin(\NA((X, q)))$ is actually modest.
Therefore, we can apply Theorem~\ref{theor-main-construction} to get that the norm given by \eqref{eq:geom-def-Read-norm} is a Read norm.
\end{prop}

This is just a consequence of Sobczyk's theorem (see \cite[2.5.8]{Albiac-Kalton}) and Proposition~\ref{prop_modest_sum}.

Our next aim is to give geometric properties of the Read norms that we have constructed in this section, which extends those results given in \cite{KadetsLopezMartin} for the original Read space.

The first result contains all the geometric properties  of the Read
norms in Theorem~\ref{theor-main-application} and
Proposition~\ref{prop-separablecontainingc0} we know about.

\begin{prop}\label{prop-geom-main-c_0}
Let $X$ be a Banach space containing $c_0$ and having a countable norming system of functionals. Then, for every $0<\eps<2$, there is an equivalent Read norm $p_\eps$ on $X$ satisfying the following:
\begin{enumerate}
    \item[(a)] $(X,p_\eps)$ is strictly convex;
    \item[(b)] every convex combination of slices of the unit ball of
    $(X,p_\eps)$ has diameter $\geq 2-\eps$, so
    every relatively weakly open subset of the unit ball of $(X,p_\eps)$
    has diameter $\geq 2-\eps$;
    \item[(c)] the norm of $(X,p_\eps)^*$ is $(2-\eps)$-rough.
\end{enumerate}
Moreover, if $X^*$ is separable, then
\begin{enumerate}
    \item[(d)] $(X,p_\eps)^{**}$ is strictly convex, so $(X,p_\eps)^*$ is smooth.
\end{enumerate}
\end{prop}

We need a couple of preliminary results for the proof  which have
their own interest. The first is surely known, but we include an
elementary proof for the sake of completeness.

\begin{lemma}\label{lemma-geom-elementary}\mbox{}
\begin{enumerate}
\item[(a)] Let $X$ be a closed subspace of $\ell_\infty$ which contains the canonical copy of $c_0$. Then, given $x\in B_X$ there are two sequences $\{y_n\}$, $\{z_n\}$ in $S_X$ that both converge weakly to $x$ and such that $e_n^*(y_n-z_n)=2$ for every $n\in \N$, where $e_n^*$ denotes
 the  $n$-th coordinate functional on $X$.
\item[(b)] Let $X$ be a Banach space such that $X=c_0\oplus_\infty Z$
  for some closed subspace $Z$. Then there is a sequence $\{f_n\}$ in
  $S_{X^*}$ such that given $x\in B_X$ there are two sequences
  $\{y_n\}$, $\{z_n\}$ in $S_X$ which converge weakly to $x$ and such
  that $f_n(y_n-z_n)=2$ for every $n\in \N$.
\end{enumerate}
\end{lemma}

\begin{proof}
For the first part, just define $y_n= x + (1-x(n))e_n$ and $z_n = x - (1 + x(n))e_n$ for every $n\in \N$, where $e_n$ is the $n$-th element of the canonical basis of $c_0$.
Then, $\{y_n\}$, $\{z_n\}$ are contained in $S_X$, both converge weakly to $x$, and $e_n^*(y_n-z_n)=2$ for every $n\in \N$.

The second part is equally easy: consider $f_n=(e_n^*,0)\in X^*$ for every $n\in \N$. Given $x=(u,z)$ with $u\in B_{c_0}$ and $z\in B_Z$, the sequences $$\{(u + (1-u(n))e_n,z)\}\qquad \text{and} \qquad \{(u - (1+u(n))e_n,z)\}$$ fulfill all of our requirements.
\end{proof}

The next preliminary result allows to transfer properties of a given norm to the norm constructed by \eqref{eq:geom-def-Read-norm}.

\begin{lemma}\label{Lemma-geom-transfer-2-eps}
Let $X$ be a Banach space and suppose that there is a sequence
$\{f_n\}$ in $S_{X^*}$ such that for every $x\in B_X$, there are two
sequences $\{y_n\}$, $\{z_n\}$ in $S_X$ which  converge weakly to $x$
and such that $f_n(y_n-z_n)=2$ for every $n\in \N$. Let
$R\colon X\longrightarrow Y$ be a compact operator from $X$ to some  Banach
space $Y$ and define an equivalent norm on $X$ by
$$
\Norm x\Norm = \|x\|_X + \|R(x)\|_Y \qquad (x\in X).
$$
Then, there is a sequence $\{g_n\}$ in the unit sphere of $(X,\Norm
\cdot\Norm )^*$ such that given $x\in X$ with $\Norm x\Norm =1$, there
exist two sequences $\{\tilde y_n\}$, $\{\tilde z_n\}$ in the unit
ball of $(X,\Norm \cdot\Norm )$  that both converge weakly to $x$ and
such that $\lim_n g_n(\tilde y_n - \tilde z_n)\geq 2(1+\|R\|)^{-1}$.
\end{lemma}

\begin{proof}
We have that $$1=\Norm x\Norm =\|x\|_X + \|R(x)\|_Y \leq (1 +
\|R\|)\|x\|_X,$$ so $\|x\|_X\geq (1+\|R\|)^{-1}$. By hypothesis, we
may take two sequences $\{y_n\}$, $\{z_n\}$ in $X$ both converging
weakly to $x$ and a sequence $\{f_n\}$ in $S_{X^*}$
such that $f_n(y_n-z_n)=2\|x\|$, $\|y_n\|=\|z_n\|=\|x\|$ and $\|y_n-z_n\|=2\|x\|$ for every $n\in \N$. As $R$ is compact, we have that $\lim R y_n=\lim R z_n=Rx$, so
$$
\lim \|R y_n\| = \lim \|R z_n\|=\|Rx\| \quad \text{and} \quad \lim \|R(y_n - z_n)\|=0.
$$
Therefore,
$\lim \Norm y_n\Norm =\lim \Norm z_n\Norm =1$ and $\lim \Norm y_n - z_n\Norm \geq 2\|x\|$.  Also,  $\Norm f_n\Norm ^*\leq\Vert f_n\Vert^*=1$.
Finally, the sequences $\tilde y_n =\Norm y_n\Norm ^{-1}y_n$, $\tilde z_n = \Norm z_n\Norm ^{-1}z_n$ and $g_n=\Norm f_n\Norm ^{-1}f_n$ fulfill all of our requirements.
\end{proof}

We are now ready to give the proof of Proposition~\ref{prop-geom-main-c_0}.

\begin{proof}[Proof of Proposition~\ref{prop-geom-main-c_0}]
We start by using Lemma~\ref{lemma-c_0 in_ellinftyl} and (the proof of) Theorem~\ref{theor-main-application} to get an equivalent norm on $X$ such that $c_0\subset X \subset \ell_\infty$ isometrically, where $c_0$ is the canonical copy, and such that $\Lin(\NA(X))$ is weak-star modest. Next, for $0<\eps<2$, we consider an operator $R_\eps$ defined by \eqref{eq:geom-def-operator-R} from Remark~\ref{remark-resume} with $\|R_\eps\|<\frac{\eps}{2-\eps}$, and consider the norm
$$
p_\eps(x)=\|x\| + \|R_\eps(x)\|_{\ell_1} \qquad (x\in X),
$$
which is a Read norm. By Proposition~\ref{prop-geom-strictly-convex},
$(X,p_\eps)$ is strictly convex, so this gives~(a). To get (b), we
just have to apply Lemmas~\ref{lemma-geom-elementary}
and~\ref{Lemma-geom-transfer-2-eps}. Indeed, let $\{g_n\}$
be the sequence in the unit sphere of $(X, p_\eps)^*$ given by
Lemma~\ref{Lemma-geom-transfer-2-eps}. Consider $C=\sum_{i=1}^N
t_iS_i$, a convex combination of slices in the unit ball of
$(X,p_\eps)$, and $x_0\in C$. We write $x_0=\sum_{i=1}^N t_i x_i$
where $x_i\in S_i$ for every $i$. There is no loss of generality if we
assume that $p_\eps(x_i)\geq 1-\delta$ for every $i$, where $\delta$
is a positive number as small as we want. By using
Lemma~\ref{Lemma-geom-transfer-2-eps} again, we get that for every $i$
there are sequences $\{y_n^{i}\}$ and $\{z_n^{i}\}$ in the unit ball
of $(X, p_\eps)$ both weakly converging to $x_i$ and such that $\lim_n
g_n(y_n^{i}-z_n^{i})\geq (1-\delta)(2-\eps) $. Therefore, for large
enough $n$,  we have that $\sum_{i=1}^N t_i y_n^{i}$,
$\sum_{i=1}^N t_i z_n^{i}$ are elements in $C$ with
distance, at least, $(1-2\delta)(2-\eps)$. As $\delta$ is arbitrary,
we conclude that the diameter of $C$ is, at least, $2-\eps$. Finally,
every relatively weakly open subset of a unit ball contains a convex
combination of slices (a result due to Bourgain, see \cite[Lemme~5.3]{Bourgain-lemma}), and this gives the last part of~(b).

Item (c) is a consequence of (b) by using \cite[Proposition~I.1.11]{D-G-Z}.

If $X^*$ is separable, we may suppose that $X=c_0\oplus_\infty Z$ for some Banach space $Z$ and we use Proposition~\ref{prop-separablecontainingc0} to get that this norm makes $\Lin(\NA(X))$ modest. Now, for $0<\eps<2$, we follow the same process as before to construct the norm $p_\eps$. Again, Proposition~\ref{prop-geom-strictly-convex} gives (a) and Lemmas~\ref{lemma-geom-elementary} and~\ref{Lemma-geom-transfer-2-eps} give (b), and \cite[Proposition~I.1.11]{D-G-Z} gives (c) from~(b). Finally, (d) is a consequence of
Proposition~\ref{prop-geom-strictly-convex} since now $\Lin(\NA(X))$ is actually modest.
\end{proof}

In the separable case, we may get Read norms with better properties by using a convenient renorming from \cite{DebsGodSR} which was used in \cite{KadetsLopezMartin} for the original Read norm.

\begin{prop}\label{prop-geom-separable-smooth}
Let $X$ be a separable Banach space containing $c_0$. Then, for every $0<\eps<2$, there is an equivalent Read norm $q_\eps$ on $X$ such that:
\begin{enumerate}
    \item[(a)] $(X,q_\eps)$ is strictly convex;
    \item[(b)] $(X,q_\eps)^*$ is strictly convex, so $(X, q_\eps)$ is smooth;
    \item[(c)] $(X,q_\eps)^*$ is $(2-\eps)$-rough, equivalently, every
    slice of the unit ball of $(X,q_\eps)$ has diameter $\geq 2-\eps$.
\end{enumerate}
Moreover, if $X^*$ is separable, then
\begin{enumerate}
  \item[(d)] $(X,q_\eps)^{**}$ is strictly convex, so $(X,q_\eps)^*$ is smooth.
\end{enumerate}
\end{prop}

\begin{proof}
We fix a dense subset $\{x_n\colon n\in \N\}$ of $B_X$, and for every $0<\rho<2$, we define the bounded linear operator $S_\rho\colon \ell_2\longrightarrow X$ by $S_\rho(\{a_n\})=\rho\sum_{n=1}^\infty \tfrac{a_n}{2^n}\,x_n$ for every $\{a_n\}\in \ell_2$, which satisfies that $\|S_\rho\|\leq \rho$. For $0<\eps<2$, we take $0<\eps'<\eps$ and $\rho>0$ such that $(2-\eps')(1+\rho)^{-1}>2-\eps$, we consider the norm $p_{\eps'}$ from Proposition~\ref{prop-geom-main-c_0}, and we define the equivalent norm $q_\eps$ on $X$ to be the one for which
$$
B_{(X,q_\eps)}= B_{(X,p_{\eps'})} + S_\rho(B_{\ell_2}).
$$
First, Lemma~\ref{lemma-prelim-sum-of-balls}(d) gives that
$\NA(X,q_\eps)=\NA(X,p_{\eps'})$ and so  $q_\eps$ is a Read norm.
It follows from Lemma~\ref{lemma-prelim-sum-of-balls}(a) that
$$
q_\eps(f)=p_{\eps'}(f) + \|S_\rho^*(f)\|_2
$$
for every $f\in X^*$. As $\ell_2$ is strictly convex and $T^*$ is
one-to-one, it follows from Lemma~\ref{lemma-prelim-sumofnorms}(d)
that $(X,q_\eps)^*$ is strictly convex, so $(X,q_\eps)$ is smooth,
giving (b). Lemma~\ref{lemma-prelim-sum-of-balls}(c) gives that
$(X,q_\eps)$ is strictly convex since both $(X,p_{\eps'})$ and
$\ell_2$ are; hence (a) holds. Finally, we know from
Proposition~\ref{prop-geom-main-c_0} that $(X,p_{\eps'})^*$ is
$(2-\eps')$-rough, and then Lemma~\ref{lemma-prelim-sumofnorms}(e)
gives that $(X,q_\eps)$ is $(2-\eps')(1+\rho)^{-1}$-rough, which gives
the first part of (c) due to the way in which we have chosen the
constants $\eps'$ and~$\rho$. Finally, the second part of (c) is
equivalent to the first one by \cite[Proposition~I.1.11]{D-G-Z}.

If moreover $X^*$ is separable, as $B_{(X,q_\eps)^{**}}= B_{(X,p_{\eps'})^{**}} + J_X(S_\rho(B_{\ell_2}))$ by Lemma~\ref{lemma-prelim-sum-of-balls}(b) and $(X,p_{\eps'})^{**}$ is strictly convex by Proposition~\ref{prop-geom-main-c_0}, the strict convexity of $(X,q_{\eps})^{**}$ follows from Lemma~\ref{lemma-prelim-sum-of-balls}(c).
\end{proof}

\section{Limits of our construction}\label{sect-limits}

The main open problem related to Read norms is the following one.

\begin{problem} \label{prob5.1}
Does every non-reflexive separable Banach space admit an equivalent norm such that the set of norm attaining functionals contains no linear subspaces of dimension two?
\end{problem}

Remark that for non-reflexive non-separable  Banach spaces the answer
to the above problem is negative. Indeed, every renorming $E$ of
$\ell_\infty(\Gamma)$ with uncountable $\Gamma$ contains an isometric
copy of $\ell_\infty(\N)$ \cite[Corollary on p.~207]{Partington}. This
copy is one-complemented, so $\NA(E) \supset \NA(\ell_\infty(\N))$,
which in turn contains an infinite-dimensional linear subspace,
viz., $\ell_1(\N)$.

Taking into account that $\ell_\infty(\Gamma)$ is a $C(K)$ space, it is natural to ask the following question.

\begin{problem} \label{probC(K)}
What is the description of those compacts $K$ for which the corresponding $C(K)$  admits an equivalent norm in which the set of norm attaining functionals contains no linear subspaces of dimension two?
\end{problem}

We do not know whether the answer to Problem~\ref{prob5.1} is positive, but we would like to discuss the reasons why our construction cannot provide such a positive answer.

Observe that the key in our construction is that $X^*\setminus \Lin(\NA(X))$ is big enough to contain a weak-star dense separable operator range. It is known that this is not possible for Banach spaces with the Radon-Nikod\'{y}m property or with an almost LUR norm, as the following result of Bandyopadhyay and Godefroy shows.

\begin{prop}[\mbox{\cite[Proposition~2.23]{Band-Godefroy}}]\label{prop-remarks-godefroy}
Let $X$ be a Banach space with the Radon-Nikod\'{y}m property or with an almost LUR norm. Then $\Lin(\NA(X))=X^*$.
\end{prop}

Therefore, the main open question related to our construction is the following one.

\begin{problem}
Does every Banach space with weak-star separable dual and failing the Radon-Nikod\'{y}m property admit an equivalent norm for which the linear span of the set of norm attaining functionals is weak-star modest in the dual space?
\end{problem}

We don't even know the answer  for the space $L_1[0,1]$.

Let us comment that the proof of
Proposition~\ref{prop-remarks-godefroy} is a consequence of the fact
that $\NA(X)$ contains a dense $G_\delta$ subset of $X^*$ when $X$ has
the Radon-Nikod\'{y}m property (see Theorem~8 in \cite{Bou}, for
instance) or $X$ has an almost LUR norm (\cite{Band-Godefroy}), so
$\NA(X)$ is residual in this case. Actually, this latter
hypothesis is  sufficient to get that $\NA(X)-\NA(X)=X^*$ from the
Baire category theorem. We include the next result, which is contained
in the proof of \cite[Proposition~2.23]{Band-Godefroy}, for
completeness.

\begin{prop}[\mbox{\cite[included in the proof of Proposition~2.23]{Band-Godefroy}}]  \label{prop-NA-residua}
Let $X$ be a Banach space. If $B$ is a residual subset of $X$, then $B-B=X$ and so $\Lin(B)=X$. In particular, if $\NA(X)$ is residual in $X^*$ then $\Lin(\NA(X))=X^*$.
\end{prop}

\begin{proof}
We just have to note that for every $x\in X$, $\bigl(x+B\bigr)\cap B$ is not empty since, otherwise, the second category set $x+B$ would be contained in the first category set $X\setminus B$, which is impossible.
\end{proof}

Let us comment that the converse result to the above one is not true: for $X=L_1[0,1]$, $\NA(X)$ is of the Baire first category (so it cannot be residual), but $\Lin(\NA(X))=X^*$ (a description of $\NA(L_1[0,1])$ can be found in \cite[Lemma~2.6]{AcostaAizpuruAronGarcia}). Therefore, our construction is not applicable to $X=L_1[0,1]$ in its usual norm, but we do not know whether it could be the case in some renorming. Actually, it is known that every separable Banach space failing the Radon-Nikod\'{y}m property can be renormed in such a way that the set of norm attaining functionals is of the first Baire category (see the proof of \cite[Theorem~3.5.5]{Bourgin}) but, as the previous example shows, this does not imply that the linear span of the set of norm attaining functionals is also of the first Baire category.

Remark also that a similar argument to the proof of Proposition~\ref{prop-NA-residua} can give us the following curious result.

\begin{prop}  \label{prop-NA-residua-rational}
Let $X$ be an infinite-dimensional Banach space. If $B$ is a residual
subset of $X$ such that $tx\in B$ for every $x\in B$ and $t\in \Q$, then $B$  contains  an infinite sequence of linearly independent elements whose linear span over the field $\Q$  lies in $B$. In particular, if $\NA(X)$ is residual in $X^*$, then  $\NA(X)$ contains the $\Q$-linear span of an $\R$-linearly independent infinite sequence.
\end{prop}

\begin{proof}
Take   $0\neq x_1\in B$. Assume, inductively,  that linearly independent elements $x_1,\ldots ,x_n\in B$ have been constructed so that the set $\{x_1,\ldots ,x_n\}$ is linearly independent and the $\Q$-linear span of the set  $\{x_1,\ldots ,x_n\}$ lies in $B$. Consider
$$
E = \bigcap_{r_1,\ldots ,r_n \in \Q}\Bigl(B -\sum_{i=1}^n r_i x_i^*\Bigr).
$$
This  is a residual subset of $B$, so it contains an element $x_{n+1}\in B$
which is linearly independent from the set $\{x_1,\ldots
,x_n\}$. Indeed, if not then $Z:=\Lin(x_1,\ldots ,x_n)$ contains $E$
and is hence residual;
but $Z$ is a nowhere dense set, being a closed and proper subspace of
$X$, which is impossible by the Baire category theorem. According to
the definition of $E$, the condition $x_{n+1}\in E$ implies that
$x_{n+1}+\sum_{i=1}^n r_i x_n\in B$ for every $r_1,\ldots ,r_n\in
\Q$. Thus we get the required infinite sequence $\{x_n\}$ of linearly
independent elements in $B$.
\end{proof}

With the above result in mind, which  can be applied to $\NA(X)$ for
Banach spaces $X$ with the Radon-Nikod\'ym property or with an almost
LUR norm, it would be nice to know if there is some Banach space $X$
so that $\NA(X)$ is residual, but still $\NA(X)$ does not contain
two-dimensional subspaces. Remark also that as a consequence of
Proposition~\ref{prop-NA-residua} in combination with
Theorem~\ref{theor-main-construction}, for a general Banach space $X$,
if $\Lin(\NA(X))$ is weak-star modest then we again get  that $\NA(X)$
is not residual.

Also,  the following result of Fonf and Lindenstrauss
\cite[Theorem~4.3]{FonfLindenstrauss}  is  worth mentioning: for every non-reflexive Banach space $X$, $X^*\setminus \NA(X)$ is not a subset of a proper operator range, equivalently, $\Lin\bigl(X^*\setminus \NA(X)\bigr)$ is dense and barrelled (use \cite[Theorem~1.1]{Nygaard} for the equivalence).

On the other hand, for separable Banach spaces, if
$\Lin(\NA(X))$ is modest (or weak-star modest), then $\Lin(\NA(X))$
has to be of the first Baire category, as we may prove using a theorem
of Banach.

\begin{prop}\label{prop-remarks-Pettis}
Let $X$ be a separable Banach space. If $\Lin(\NA(X))$ is
of the second Baire category in $X^*$, then $\Lin(\NA(X))=X^*$.
\end{prop}

\begin{proof}
The argument relies on notions and results from descriptive set theory
that we'll recall in the course of the proof. A Polish space is a
completely metrisable separable topological space, and an analytic set
is a subset of a topological space which is a continuous image of some
Polish space. Since $X$ is separable, $\NA(X)$ is an analytic subset
of $X^*$ equipped with the weak-star topology; see
\cite[p.~221]{Kaufman}.
We shall argue that $\Lin(\NA(X))$ is analytic as well.

 For $n\in\N$ define  $f_n\colon \NA(X)^n \times \R^n\to X^*$ by
 $$
 f_n(x_1^*,\dots,x_n^*, t_1,\dots,t_n) = \sum_{k=1}^n t_k x_k^*;
 $$
 then
 $$
 \Lin (\NA(X)) = \bigcup_{n\in\N} f_n( \NA(X)^n \times \R^n ).
 $$
 Now the class of analytic sets is closed under taking finite (even
 countable) products, continuous images, and countable unions;
 therefore $\Lin(\NA(X))$ is indeed analytic for the weak-star
 topology.

In a Hausdorff topological space, analytic sets are known to be $F$-Souslin
\cite[Theorem~6.6.8]{Boga2},
that is, they can be represented as
$$
\bigcup_\sigma \bigcap_{n=1}^\infty  F_{\sigma_1,\dots,\sigma_n}
$$
for closed sets $F_{\sigma_1,\dots,\sigma_n} $ where the union is
taken over all sequences   $\sigma = (\sigma_1,\sigma_2, \dots)$ of
positive integers. Hence $\Lin(\NA(X))$ is $F$-Souslin for the
weak-star topology and therefore also for the norm topology.

The next piece of information that we need concerns the Baire
property. A subset of a topological space has the Baire property if it
differs from an open set by a set of the first category; that is, if
it can be written in the form $G \mathbin{\Delta} M$ with $G$ open and $M$ of
the first category where $\Delta$ denotes the symmetric difference. In a
Hausdorff space, every $F$-Souslin set has
the Baire property; see \cite[Corollary~2.9.4]{RJ}. Consequently,
$\Lin(\NA(X))$ has the Baire property for the norm topology.

Finally we apply a theorem due to Banach
(\cite[Th\'eor\`eme~2, Chapitre~1]{Banach}
or \cite[p.~92]{Kelley-Namioka}) which assures that,
in a Banach space,
a linear subspace of the second Baire category which satisfies the
Baire property is the whole space.
\end{proof}

We are grateful to W.~Moors for indicating the above argument to us;
in a preliminary version of this paper we had to make the far stronger
assumption that $X^*$ is separable.

We do not know whether separability can be dropped from
Proposition~\ref{prop-remarks-Pettis}.

As a consequence of Proposition~\ref{prop-remarks-Pettis}, if $X$ is
separable and $\Lin(\NA(X))$ is weak-star modest, then $\Lin(\NA(X))$
has to be of the first Baire category. We do not know whether the
converse is also true, but there is a partial answer.

\begin{prop}
Let $X$ be a Banach space such that $X^*$ is separable. If
$\Lin(\NA(X))$ is not barrelled, then $\Lin(\NA(X))$ is modest (and
so, $X$ admits an equivalent Read norm).
\end{prop}

\begin{proof}
By \cite[Theorem~15.2.1]{Wilansky}, it follows that $\Lin(\NA(X))$ is
contained in a (dense) proper operator range. Now, \cite[Theorem~1]{Shevchik}
shows that there is a dense operator range $Y$ in $X^*$ such that
$Y\cap \Lin(\NA(X))=\{0\}$, that is, $\Lin(\NA(X))$ is modest.
\end{proof}

We have to mention that this result does not produce new examples of
spaces which admit Read norms, as the following result by Fonf shows:
if $\Lin(\NA(X))$ is not barrelled, then $X$ contains a copy of $ c_0$ (see \cite[Theorem~2.6]{Nygaard} for a version of Fonf's result using this language).

Let us note that the task to find a Banach space $X$ with $X^*$
separable such that $\Lin(\NA(X))$ is weak-star modest and $X$ does
not contain $c_0$, requires to find a Banach space $X$ such that
$\Lin(\NA(X))$ is of the first Baire category and barrelled.

Finally, it would be interesting to find examples of Banach spaces $X$
for which $\Lin(\NA(X))$ is modest (or weak-star modest) in their
usual norm, as it happens with $c_0$. Another example of this kind is
given in the papers \cite{Aco-Edinburgh,Aco-contemporary} by Acosta:
let $w\in \ell_2\setminus \ell_1$ with $0<w_n<1$ for all $n$ and
consider the space $Z$ of sequences $z$ of scalars for which
$$
\|z\|:=\|(1-w)z\|_\infty \,+\,\|wz\|_{\ell_1} <\infty
$$
endowed with this function as norm. Then, the sequence $\{e_n\}$ of
unit vectors is a $1$-unconditional basic sequence of $Z^*$ whose
closed linear span $X(w)$  is an isometric predual of
$Z$ for which $\{e_n\}$ is a $1$-unconditional basis whose
biorthogonal basis $\{e_n^*\}$ is again the canonical unit vector
basis \cite[Lemma~2.1]{Aco-Edinburgh}.
Then, $\Lin(\NA(X(w)))$ is modest in $X^*$. Indeed, it is shown
in \cite[Lemma~2.2]{Aco-contemporary} that if $x^*\in X(w)^*$ belongs
to $\NA(X(w)^*)$, then $w\chi_{\supp(x^*)}\in \ell_1$; if we consider
the bounded linear operator $T\colon \ell_1 \longrightarrow X(w)^*$ given by
$T(e_n)=e_n^*$ for every $n\in \N$,  it follows, as $w\notin \ell_1$,
that $T(Y)\cap \Lin(\NA(X))=\{0\}$ where $Y$ is the operator range of
$\ell_1$ given in Lemma~\ref{Lem-op-im-inell1}. Let us observe that
$X(w)$ contains a copy of $c_0$ (since the basis is unconditional and
shrinking and the space is not reflexive), so we already know from
Section~\ref{sect:containing_c_0} that it admits an equivalent Read
norm.

\end{document}